\documentclass[reqno]{amsart}

\usepackage{enumerate,amsmath,amssymb}

\newtheorem{lem}{Lemma}
\newtheorem{cor}[lem]{Corollary}
\newtheorem{defi}[lem]{Definition}
\newtheorem{thm}[lem]{Theorem}

\theoremstyle{definition}

\newtheorem{rem}[lem]{Remark}

\newcommand{\ov}{\overline}
\newcommand{\s}{\mathrm{scal\,}}
\newcommand{\Hy}{\mathbb{H}}
\newcommand{\vo}{\mathrm{dvol}}
\newcommand{\bqw}{\begin{equation*}}
\newcommand{\eqw}{\end{equation*}}
\newcommand{\bq}{\begin{equation}}
\newcommand{\eq}{\end{equation}}
\newcommand{\vol}{\mathrm{vol}}
\newcommand{\R}{\mathbb{R}}
\newcommand{\supp}{\mathrm{supp\ }}
\renewcommand{\d}{\mathrm{d}}

\begin{document}

\title{The Yamabe equation on complete manifolds with finite volume}
\author{Nadine Gro\ss e}
\subjclass[2000]{53C21, 53A30, 35R01}

\date{\today}

\keywords{Yamabe problem, scalar curvature, complete manifolds of finite volume}

\thanks{I like to thank the Hausdorffcenter in Bonn where parts of this paper were written down for its hospitality.}

\date{}

\begin{abstract}
We prove the existence of a solution of the Yamabe equation on complete manifolds with finite volume and positive Yamabe invariant. In order to circumvent the standard methods on closed manifolds which heavily rely on global (compact) Sobolev embeddings we approximate the solution by eigenfunctions of certain conformal complete metrics. 

This also gives rise to a new proof of the well-known result for closed manifolds and positive Yamabe invariant. 
\end{abstract}

\maketitle

\section{Introduction}

Yamabe examined whether a closed $n$-dimensional Riemannian manifold $(M,g)$ $(n\geq 3)$ possesses a metric $\ov{g}$ conformal to $g$ with constant scalar curvature.  His striking idea was the consideration of the so-called Yamabe invariant, see Definition \ref{Yam_variation}, which gave the possibility to view the question as a variational problem. Works from Aubin \cite{Au76}, Schoen \cite{Sch84} and Trudinger \cite{Tr} answered the question of Yamabe affirmatively. 

There are several possibilities to generalize the Yamabe problem to open (i.e. noncompact and without boundary) manifolds. 

One possibility is simply to pose the same question as Yamabe did. On open manifolds, this gives much more freedom. We want to make this more precise by comparing to the closed case:

On closed manifolds, if the Yamabe invariant $Q$ is nonpositive, then every conformal metric having constant scalar curvature $c$ and unit volume fulfills $c=Q$. In case that $Q<0$ this conformal metric is even unique. If $Q >0$ and there is a conformal metric with constant scalar curvature $c$, one see immediately that $c\geq Q$. But nevertheless, on closed manifolds in all cases a conformal metric with constant scalar curvature has the same sign as the corresponding Yamabe invariant. 

On open manifolds, this is no longer true, an easy example is given by an open ball in the Euclidean space. Its Yamabe invariant is the one of the standard sphere, but it carries conformal metrics of constant scalar curvature of all signs: The original Euclidean metric has zero scalar curvature, the spherical metric has constant positive  and the hyperbolic metric has constant negative scalar curvature. But all those metrics are conformally equivalent. 

That's why the question is often posed more restrictively. A first possibility is to fix the sign of the constant scalar curvature and/or ask additionally for completeness. This was done by many authors and many results with positive and negative answers were obtained, see for example \cite{AMcO88}, \cite{Jin88}, \cite{Zhang97}.

The second is to stick to the original Yamabe problem and ask for solutions to the Euler-Lagrange equation of the Yamabe problem, i.e. unit volume metrics with constant scalar curvature $Q$. This version was studied for example for manifolds with bounded geometry and positive scalar curvature in \cite{Kim96} using a compact exhaustion of the open manifold and for manifolds bounded geometry and positive Yamabe invariant in \cite{NG3} using weighted Sobolev embeddings. 

In this paper we consider the second type of a noncompact Yamabe problem on complete manifolds of finite volume. 

Let $(M^n,g)$ be an $n$-dimensional complete connected Riemannian manifold of finite volume and $n\geq 3$. Let $L_g=a_n\Delta_g +\s_g$ be the conformal Laplacian where $\s_g$ is the scalar curvature of the metric $g$ and $a_n=4\frac{n-1}{n-2}$.

\begin{defi}\label{Yam_variation}
The Yamabe invariant of $(M,g)$ is given by

\[ Q(M,g)=\inf\Bigg\{ Q_g(v):=\frac{\int_M vL_g v\vo_g}{\Vert v\Vert_{L^p(g)}^2}\ \Bigg|\ v\in C_c^\infty(M), v\ne 0\Bigg\} \] 

where $p=\frac{2n}{n-2}$ and $C_c^\infty(M)$ denotes the set of compactly supported real valued functions on $M$.
\end{defi}

$Q$ is conformally invariant which is seen from the conformal transformation formula of the conformal Laplacian: For $\ov{g}=f^2g$ where $f\in C_{>0}^\infty(M)$ is a smooth positive real function on $M$ we have
\bqw
L_{\ov{g}} \ov{v} =  f^{} L_g v\ {\rm where\ } \ov{v}=f^{-\frac{n-2}{2}} v.\eqw

The Yamabe invariant is given as a variational problem. Its Euler-Lagrange equation is  
\bq\label{EL-eq} L_g v= Qv^{p-1}\quad v\in H_1^2,\ \Vert v\Vert_{L^p(g)}=1.\eq

The aim of this paper is to study the existence of a smooth positive solution of \eqref{EL-eq} for complete manifolds of finite volume. 

The standard proof for the Yamabe problem on closed manifolds heavily relies on the existence of compact Sobolev embeddings. On complete open manifolds of finite volume there do not even exist continuous Sobolev embeddings $H_1^2\hookrightarrow L^p$, \cite[Lem. 3.2]{Heb}. That's why we will use a different approach by approximating the desired solution by certain eigenfunctions of conformal metrics, cp. Section \ref{Q_eigenv}.

In the standard proof on closed manifolds one uses the subcritical Yamabe problem to get solutions of differential equations that are somehow 'near' to the desired Euler-Lagrange equations. This allows to show converges of a sequence of those solutions which then serve as test functions for the critical problem. In our approach here the eigenfunctions will play the role of these subcritical solutions and we obtain

\begin{thm}\label{main}
Let $(M,g)$ be an open complete manifold of finite volume with $0<Q(M,g)< \ov{Q}(M,g)$ and $\Vert (\s_g)_-\Vert_\frac{n}{2}<\infty$ where $(\s_g)_-:=-\min \{\s, 0\}$.

Then, there exists a smooth positive solution $v\in H_1^2$ of $L_gv=Qv^{p-1}$ with $\Vert v\Vert_{L^p(g)}=1$.
\end{thm}

$\ov{Q}(M,g)$ is the Yamabe invariant at infinity, see Definition \ref{def_lminf}, and replaces $Q(S^n)$ that appears at this point in the closed case, cf. Remark \ref{Yam_inf_closed}.

The non-existence of a continuous Sobolev embedding $H_1^2\hookrightarrow L^p$ has the following straightforward implications:
If $Q>0$, $\s_g$ cannot be bounded from above. Moreover, if $v$ is a solution as in Theorem \ref{main}, $\ov{g}=v^\frac{2n}{n-2}g$ is a metric with finite volume and constant scalar curvature and for all $v\in C_c^\infty(M)$ $\Vert v\Vert_{L^p(\ov{g})}\leq (\max\{ a_n, Q\})^{\frac{1}{2}} \Vert v\Vert_{H_1^2(\ov{g})}$. Thus, $\ov{g}$ cannot be complete.

The method used to prove Theorem \ref{main} also gives rise to a different proof for the closed case with positive Yamabe invariant, see Theorem \ref{main_closed}. 
Moreover, the method can be adapted to similar contexts, e.g one can obtain similar results for the spinorial Yamabe invariant, cf. \cite{NG6}. 

\section{Preliminaries}

In this section we collect some facts on the Yamabe invariant.

\begin{rem}\label{def_sol_Yam} On complete manifolds
 instead of taking the infimum over $C_c^\infty$ in the definition of the Yamabe invariant \ref{Yam_variation} one could as well take the infimum over $v\in L^2 \cap H^2_{1,loc}$ with $\Big| \int_M vL_g v\vo_g\Big|<\infty$. This is seen when considering $Q_g(\eta_i v)$ for suitable cut-off functions $\eta_i$ with $\eta_i\to 1$. 
\end{rem}

\begin{defi}\label{def_lminf}\ \cite{Kim00}
Let $(M,g)$ be an open $n$-dimensional manifold with a compact exhaustion $K_i$ fulfilling $K_i\subset K_{i+1}\varsubsetneq M$ and $\cup_i K_i= M$. Then the Yamabe invariant at infinity is defined as 
\[ \ov{Q}(M,g):= \lim_{i\to \infty} Q(M\setminus K_i, g).\] 
\end{defi}

Note that $Q(M\setminus K_i, g)\leq Q(M\setminus K_{i+1}, g)$ since when considering only a subset less test functions can be used in Definition \ref{Yam_variation}. Together with $Q(M,g)\leq Q(S^n)$ \cite{Schoe} where $Q(S^n)=n(n-1)\vol(S^n)^\frac{2}{n}$ is the Yamabe invariant of the sphere with the standard metric the sequence $Q(M\setminus K_i, g)$ is monotonically increasing and bounded. Thus, $\ov{Q}$ always exists and it holds $\ov{Q}(M,g)\leq Q(S^n)$. Furthermore, $\ov{Q}$ does not depend on the choice of the sequence $K_i$.

\begin{rem}\label{Yam_inf_closed}
We note that the condition $\ov{Q}(M,g)<Q(S^n)$ in Theorem \ref{main} replaces $Q(M,g)< Q(S^n)$ that appears in the closed case. This can be seen since for $p\in M$ we have 
$Q(M,g)=Q(M\setminus \{p\},g)$ \cite[Lem. 2.1]{Schoe} and $\ov{Q}(M\setminus\{p\},g)=\lim_{i\to \infty} Q(B_\epsilon(p), g) = Q(S^n)$ where $B_\epsilon(p)$ is a ball around $p$ with radius $\epsilon$.
\end{rem}

The blow-up argument in the standard proof of the Yamabe problem \cite{SY} which rules out concentration phenomena at a fixed point
shows that for fixed  $x\in M$ $Q(B_\epsilon(x),g)\to Q(\R^n,g_E)=Q(S^n, g_{st})$ as $\epsilon\to 0$. We will need the following slight generalization:

\begin{lem}\label{local_comp}
For all compact subsets $U\subset M$ and $\delta>0$ there is an $\epsilon=\epsilon(U,\delta) > 0$ such that for all $x\in U$: $Q(B_\epsilon(x),g)\geq Q(S^n)-\delta$.
\end{lem}

\begin{proof} Let $U$ and $\delta$ be fixed. Then for each $x\in U$ let $\epsilon(x)$ be the maximal radius such that $Q(B_\epsilon(x),g)\geq Q(S^n)-\delta$ is fulfilled. Set $\epsilon= \inf_{X\in U} \epsilon(x)$. Suppose $\epsilon=0$. Then there is a sequence $x_i\in U$ with $\epsilon(x_i)\to 0$. Since $U$ is compact, $x_i\to x\in U$. Note that on closed manifolds $Q$ depends smoothly on $g$ in the $C^2$-topology \cite[Proof of Prop. 7.2.]{Ber}. Thus, $\epsilon(x_i)\to \epsilon(x)>0$ which is a contradiction. Thus, $\epsilon>0$.
\end{proof}

\section{Nonnegative Yamabe invariants and the $L^2$-spectrum}\label{Q_eigenv}

On closed manifolds and if $Q\geq 0$, 

\[ Q(M,g)= \inf \{ \mu(L_{\ov{g}})\ |\ \ov{g}\in [g] \}\]
where $\mu(L_g)$ is the lowest eigenvalue of the conformal Laplacian $L_g$ and $[g]:= \{ \ov{g}=f^2g\ | f\in C_{>0}^\infty(M)\}$ denotes the conformal class of $g$.

On general manifolds the spectrum of $L_g$ does not only contain eigenvalues but there can be residual and continuous spectrum. Moreover, in general $L_g$ is even not essentially self-adjoint.

We consider
\[ \mu(L_g)= \inf \Bigg\{ \frac{\int_M vL_g v\vo_g}{\Vert v\Vert_{L^2(g)}^2}\ \Bigg|\ v\in C_c^\infty(M)\Bigg\}. \]

If $L_g$ is essentially self-adjoint, $\mu(L_g)$ is the minimum of the spectrum of $L_g$.

\begin{rem}\label{ess_sa}
If $Q\geq 0$ and $\vol(M,g)=1$, then $\int_M vL_g v\vo_g\geq Q\Vert v\Vert_p^2\geq Q\Vert v\Vert_2^2$, i.e. $\mu(L_g)\geq Q$ and $L_g$ is bounded from below. Then, $L_g$ is essentially self-adjoint on $C_c^\infty(M)$, \cite[Thm. 1.1]{Shu01} and possesses only eigenvalues and essential spectrum. Moreover, the spectrum is real.
\end{rem}

If it is clear from the context to which Riemannian manifold $(M,g)$ we refer, we abbreviate $\Vert . \Vert_s:=\Vert .\Vert_{L^s(g)}$.

\begin{lem}
Let $(M,g)$ be a Riemannian manifold with $Q\geq 0$. Then
\[ Q(M,g)= \inf \{ \mu(L_{\ov{g}})\ |\ \ov{g}\in [g], \vol(M,\ov{g})=1 \}.\]

If $(M,g)$ has additionally unit volume,
\[ Q(M,g)= \inf \{ \mu(L_{\ov{g}})\ |\ \ov{g}=f^2g, \vol(M,\ov{g})=1,\  \exists \rm{\ a\ compact\ subset\ } K_f\subset M: f|_{M\setminus K_f} =1 \}.\]
\end{lem}

If for a function $f\in C_{>0}^\infty(M)$ such a compact subset $K_f$ exists, we shortly say that $f\equiv 1$ near infinity.
The proof of the first part is the same as in the closed case. But since we are not aware of a reference we shortly give the proof.

\begin{proof} 
Without loss of generality we can assume that $g$ already has unit volume. Since $Q\geq 0$, $\int_M vL_g v\vo_g\geq 0$ for all $v\in C_c^\infty(M)$.

From Remark \ref{ess_sa} we have $\mu (L_{\ov{g}})\geq Q(M,g)$ for all conformal metrics $\ov{g}\in[g]$ with unit volume.

On the other hand, let $v_i\in C_c^\infty(M)$ be a minimizing sequence for $Q$  with $\Vert v_i\Vert_p=1$ and $\int_M v_i L_gv_i\vo_g\to Q$. Set $g_i=\left( \Vert v_i+i^{-1}\Vert_p^{-1} (v_i +i^{-1})\right)^{\frac{4}{n-2}}g$. Then, $\vol(M, g_i)= \int_M \left(\Vert v_i+i^{-1}\Vert_p^{-1} (v_i +i^{-1})\right)^p\vo_g =1$ (Note that $\Vert v_i+i^{-1}\Vert_p\leq \Vert v_i\Vert_p + i^{-1}= 1+i^{-1}$ is finite.). Moreover, 
\begin{align*}
\Vert \ov{v}_i \Vert_{L^2(g_i)}^2&= \int_M \left(\Vert v_i+i^{-1}\Vert_p^{-1} (v_i +i^{-1})\right)^\frac{4}{n-2} v_i^2\vo_g\\
&\geq \int_M (1+i^{-1})^{-\frac{4}{n-2}} (v_i +i^{-1})^\frac{4}{n-2} v_i^2\vo_g\\
&\geq (1+i^{-1})^{-\frac{4}{n-2}} \int_M  v_i^p\vo_g = (1+i^{-1})^{-\frac{4}{n-2}}.
\end{align*}

Hence, 
\[ 0\leq \mu(L_{g_i})\leq \frac{\int_M \ov{v}_iL_{g_i} \ov{v}_i\vo_{g_i}}{\Vert \ov{v}_i\Vert_{L^2(g_i)}^2 }\leq \frac{\int_M v_iL_{g} {v}_i\vo_{g}}{(1+i^{-1})^{-\frac{4}{n-2}}}\to Q(M,g)
\] as $i\to \infty$ which finishes the proof of the first claim.

Let now $(M,g)$ be complete and of finite volume and $v_i$ be the test sequence of above. Let $K_i$ be a sequence of compact subsets with $\supp v_i \subset K_i$. The value $\Vert \ov{v}_i\Vert_{L^2(g_i)}^{-2}\int_M \ov{v}_iL_{g_i} \ov{v}_i\vo_{g_i}$ only depends on the metric $g_i$ on $\supp v_i$. Thus, we can deform $g_i$ such that the conformal factor $f_i=1$ outside a compact subset $K_{f_i}$ with $K_i\subset\subset K_{f_i}\varsubsetneq M$, $f_i^2g\equiv g_i$ on $K_i$ and $\vol (f_i^2g)=1$. 

In particular, if $g$ was complete, all those $\ov{g}$ are also complete.
\end{proof}

Next we study the Yamabe invariant if essential spectrum is present.
\begin{lem} 
Let $(M,g)$ be a complete Riemannian manifold of unit volume. Let $L_g$ be essentially self-adjoint on $C_c^\infty(M)$ and let the essential spectrum of $L_g$ be non-empty. Then $Q(M,g)=\ov{Q}(M,g)\leq 0$.
\end{lem}

\begin{proof}
Let $\mu$ be in the essential spectrum of $L_g$. Then there is a sequence $v_i\in C_c^\infty(M\setminus B_i)$ where $B_i$ a ball with radius $i$ around a fixed point $z\in M$ such that $\Vert (L_g-\mu)v_i\Vert_2\to 0$, $\Vert v_i\Vert_2=1$ and $v_i \to 0$ weakly in $L^2$.
Then, using $1=\Vert v_i\Vert_2\leq \Vert v_i\Vert_p\vol(M\setminus B_i)^\frac{2}{n}$ and, thus, $\Vert v_i\Vert_p\geq 1$ we estimate
\begin{align*}
Q(M&\setminus B_i)\leq\frac{\int_M v_iL_gv_i\vo_g}{\Vert v_i\Vert_p^2}\leq \frac{\Vert L_gv_i\Vert_2\Vert v_i\Vert_2}{\Vert v_i\Vert_p^2}\leq
\frac{(\Vert (L_g-\mu)v_i\Vert_2 +|\mu|)\Vert v_i\Vert_2}{\Vert v_i\Vert_p^2}\\
&\leq
\frac{(\Vert (L_g-\mu)v_i\Vert_2 +|\mu|)\Vert v_i\Vert_p \vol (M\setminus B_i)^{\frac{2}{n}}}{\Vert v_i\Vert_p^2}\leq
(\Vert (L_g-\mu)v_i\Vert_2 +|\mu|)
\vol (M\setminus B_i)^{\frac{2}{n}}
\end{align*}
where the right hand-side goes to zero as $i\to \infty$.
\end{proof}

From that and Remark \ref{ess_sa} it follows directly
\begin{cor} \label{eig_seq}
If $(M,g)$ is a complete Riemannian manifold of unit volume with $Q(M,g)> 0$ or $\overline{Q}(M,g)>Q(M,g)=0$, there exists a sequence of $g_i=f_i^2g$ with $f_i\in C^\infty_{>0}(M)$, $\int_M f_i^n\vo_g=1$ and eigenvalues $\mu_i:=\mu (L_{g_i})\to Q$ as $i\to \infty$.
\end{cor}

\section{Proof of Theorem \ref{main}}\label{proof_of_main}

From Corollary \ref{eig_seq} we have: If $Q>0$, there exists a sequence of eigenfunctions $\ov{v}_i$ with $L_{{g}_i}\ov{v}_i=\mu_i\ov{v}_i$ and $\int_M |\ov{v}_i|^2\vo_{g_i}=1$. $\ov{v}_i$ is eigenfunction to the lowest eigenvalue $\mu_i$ of $g_i=f_i^2g$ ($f_i=1$ near infinity)  and, hence, positive. Viewing these equation w.r.t. the reference metric $g$ we obtain the following setting

\bqw
 L_g v_i =\mu_i f_i^2v_i\quad \int_M f_i^2v_i^2\vo_g=1, \int_M f_i^n\vo_g=1, \mu_i\searrow Q, v_i>0.\eqw

Firstly we note that $\int_M |\ov{v_i}|^2\vo_{g_i}=\int_M f_i^2v_i^2\vo_{g_i}=1$ and $f_i=1$ outside a compact subset implies $v_i\in L^2(g)$. 

Moreover, due to Remark \ref{def_sol_Yam} $v_i$ can serve as a test function for $Q$ and, thus, $Q\Vert v_i\Vert_p^2 \leq \int_M v_iL_gv_i\vo_g=\mu_i$. Since $Q>0$ $v_i\in L^p(g)$ and if then $i\to \infty$, we obtain $\Vert v_i\Vert_p\to 1$. Thus, $v_i$ is uniformly bounded in $L^p(g)$ and, due to the finite volume, also in $L^2(g)$. 

From 
\begin{align*}
\mu_i&=\int_M v_iL_gv_i\vo_g = a_n\Vert \d v_i\Vert_2^2+\int_M \s_g v_i^2\vo_g\\
&\geq a_n\Vert \d v_i\Vert_2^2-\int_M (\s_g)_- v_i^2\vo_g\geq a_n\Vert \d v_i\Vert_2^2-\Vert(\s_g)_-\Vert_{\frac{n}{2}} \Vert v_i\Vert_p^2
\end{align*}
 and the assumption that $\Vert(\s_g)_-\Vert_{\frac{n}{2}}<\infty$  we see that $\Vert \d v_i\Vert_2$ is also uniformly bounded. Summarizing $v_i$ is uniformly bounded in $H_1^2$ and, hence, $v_i\to v\geq 0$ weakly in $H_1^2$ and in $L^p$. Moreover, $\int_M f_i^n\vo_g=1$ implies that there is $f\in L^{n}$ such that $f_i^2\to f^2$ weakly in $L^\frac{n}{2}$.

\begin{lem}\label{lem1} Let $f_i^2\to f^2$ weakly in $L^\frac{n}{2}$ and $v_i\to v$ weakly in $H_1^2$.
\begin{itemize}
\item[i)] Then $f_i^2v_i\to f^2v$ in $L^s(U)$ for all compact subsets $U\subset M$ and $1< s < q=\frac{2n}{n+2}$. 
\item[ii)] If additionally $L_gv_i=\mu_if_i^2v_i$ and $\mu_i\to Q$, $f$ and $v$ weakly fulfill $L_gv=Qf^2v$.
\end{itemize}
\end{lem}

\begin{proof}
i) We fix $w\in L^{s^*}(U)$ with $\frac{1}{s}+\frac{1}{s^*}=1$.
Then 

\[ \Bigg|\int_U (f_i^2v_i-f^2v)w\vo_g\Bigg|\leq \int_U |f_i^2-f^2|\,|vw|\vo_g + \int_U f_i^2|v_i-v|\,|w|\vo_g\]
The weak convergence $v_i\to v$ in $H_1^2$ implies strong convergence $v_i\to v$ on $L^{p'}(U)$ for all $1\leq p'<p$.  We choose $q'$ such that $p>q'>\frac{n}{n-2}$. Then H\"older inequality implies 
\[\Vert vw\Vert_{L^\frac{n}{n-2}(U)}\leq \Vert v\Vert_{L^{q'}(U)} \Vert w\Vert_{L^{\frac{nq'}{(n-2)q'-n}}(U)}<\infty\] if $\frac{nq'}{(n-2)q'-n}\leq s^*$. The choice of $q'$ implies $p<\frac{nq'}{(n-2)q'-n}<\infty$ and, thus, $p < s^*<\infty$ and $1<s<q$.

Hence, $  \int_U |f_i^2-f^2|\, |vw|\vo_g \to 0$ as $i\to \infty$. For the second summand of the above inequality we have
\begin{align*}
\int_U f_i^2|v_i-v|\,|w|\vo_g&\leq \Vert v_i-v\Vert_{L^{q'}(U)} \Vert f_i^2w\Vert_{L^{\frac{q'}{q'-1}}(U)}\\
&\leq  \Vert v_i-v\Vert_{L^{q'}(U)} \Vert f_i\Vert_{L^n(U)}^{2} \Vert w\Vert_{L^{\frac{nq'}{(n-2)q'-n}}(U)} \leq  \Vert v_i-v\Vert_{L^{q'}(U)} \Vert w\Vert_{L^{s}(U)} \\
& \to 0\quad {\rm as\ } i\to \infty.
\end{align*}

ii) Let $w\in C_c^\infty(M)$.
\begin{align*}
\left|\int_M (L_gv-Qf^2v)w\vo_g\right|&=\left|\int_M (L_gv-L_gv_i+\mu_if_i^2v_i-Qf^2v)w\vo_g\right|\\
&\leq a_n\left|\int_M  (\d v-\d v_i)\d w\vo_g\right| + 
\left|\int_M  ( v-v_i)(\s_g w)\vo_g\right|\\
&\phantom{=} +\mu_i \left|\int_M (f_i^2v_i-f^2v)w\vo_g\right|+|Q-\mu_i|\int_M f^2|vw|\vo_g
\end{align*}
All summands on the right-hand side tend to zero as $i\to\infty$ since $v_i\to v$ weakly in $H_1^2$ (note that $\s_g w\in C_c^\infty \subset L^2$), part i) and $\mu_i\to Q$.
\end{proof}

In order to finish the proof of Theorem \ref{main} it remains to show that $f^2=v^{p-2}$ and $\Vert v\Vert_p=1$. We start with a non-vanishing result.

\begin{lem}\label{lem2} In the setting of Lemma \ref{lem1} and assuming $0<Q(M,g)<Q(S^n)$, $v$ does not vanish identically.
\end{lem}

\begin{proof} We prove by contradiction and  assume $v\equiv 0$ and, hence, $\int_U v_i^{s}\vo_g\to 0$ for all compact subsets $U\subset M$ and $1\leq s<p$.

Firstly, we want to show that then also $\int_U v_i^{p}\vo_g\to 0$ for all compact subsets $U\subset M$. For that, we assume the contrary, i.e. $\Vert v_i\Vert_{L^p(U)}> C(U)>0$ and consider small balls $B_{2\epsilon}(x)$ with $x\in M$. We choose $\epsilon$ small enough such that for all $x\in U$
$Q(B_{2\epsilon}(x))> Q(M,g)$. Due to $Q(S^n)>Q(M,g)$ and Lemma \ref{local_comp} this is always possible. Then we cover $U$ by finitely many of those balls  $B_{2\epsilon}(x)$ and define smooth cut-off functions $\eta_{\epsilon,x}$ compactly supported in $B_{2\epsilon} (x)$ that are $1$ on $B_\epsilon (x)$ and $|\d \eta_{\epsilon,x}|\leq 2\epsilon^{-1}$. Then we 
estimate 

\begin{align*}
Q(B_{2\epsilon}(x),g)&\leq \frac{\int_{B_{2\epsilon}} \eta_{\epsilon,x} v_i L_g(\eta_{\epsilon,x} v_i) \vo_g}{\left( \int_{B_{2\epsilon}(x)} \eta_{\epsilon,x}^p v_i^p\vo_g\right)^\frac{2}{p}} 
=\frac{\int_{B_{2\epsilon}} \eta_{\epsilon,x}^2 v_i L_g v_i \vo_g + a_n\int_{B_{2\epsilon}} |\d \eta_{\epsilon,x}|^2v_i^2\vo_g
}{\left( \int_{B_{2\epsilon}(x)} \eta_{\epsilon,x}^p v_i^p\vo_g\right)^\frac{2}{p}}\\
&\leq\frac{\mu_i\int_{B_{2\epsilon}(x)} \eta_{\epsilon,x}^2 f_i^2v_i^2 \vo_g + a_n\frac{4}{\epsilon^2} \int_{B_{2\epsilon}(x)} v_i^2
}{\left( \int_{B_{2\epsilon}(x)} \eta_{\epsilon,x}^p v_i^p\vo_g\right)^\frac{2}{p}} \leq \mu_i + a_n\frac{4}{\epsilon^2} \frac{\int_{B_{2\epsilon}(x)} v_i^2\vo_g}{\left( \int_{B_{\epsilon}(x)} v_i^p\vo_g\right)^\frac{2}{p}}\\
&\leq \mu_i + a_n\frac{4}{\epsilon^2C(U)^2} \int_{B_{2\epsilon}(x)} v_i^2\vo_g
\end{align*}

where in the second last step we used the H\"older inequality to estimate the summand including $\mu_i$.
If $i$ tends to $\infty$, we obtain $Q(B_{2\epsilon}(x),g)\leq Q$ which is a contradiction to $Q(B_{2\epsilon}(x))> Q$. Thus, $\Vert v_i\Vert_{L^p(U)} \to 0$ as $i\to \infty$.

Next, let $\chi_R$ be a smooth cut-off function with $\chi_R=0$ on $B_R := B_R(z)$ for a fixed $z\in M$, $\chi_R=1$ on $M\setminus B_{2R}$  and $|\d \chi_R|\leq 2R^{-1}$. Then

\begin{align*}
Q&=\lim_{i\to\infty} \int_M v_iL_g v_i \vo_g = \lim_{i\to \infty} \left(\int_M \chi_R^2v_iL_gv_i\vo_g + \mu_i\int_M (1-\chi_R^2)f_i^2v_i^2\vo_g\right)\\
&\geq \lim_{i\to \infty} \left(\int_M \chi_R^2v_iL_gv_i\vo_g + \mu_i\int_{B_R} f_i^2v_i^2 \vo_g\right)\\
&\geq \lim_{i\to \infty} \left(\int_M \chi_Rv_iL_g(\chi_Rv_i)\vo_g -  a_n\int_M |\d \chi_R|^2v_i^2\vo_g+\mu_i\int_{B_R} f_i^2v_i^2 \vo_g\right)\\
&\geq \lim_{i\to \infty} \left(Q(M\setminus B_R,g) \Vert \chi_Rv_i\Vert_p^2 - \frac{4a_n}{R^2} \Vert v_i\Vert_{L^2(B_{2R})}^2 +\mu_i\int_{B_R} f_i^2v_i^2 \vo_g\right)\\ 
&\geq \lim_{i\to \infty} \left(Q(M\setminus B_R,g) ( \Vert v_i\Vert_p - \Vert (1-\chi_R)v_i\Vert_p)^2 - \frac{4a_n}{R^2}\Vert v_i\Vert_{L^2(B_{2R})}^2  +\mu_i\int_{B_R} f_i^2v_i^2 \vo_g\right)\\ 
&\geq \lim_{i\to \infty} \left(Q(M\setminus B_R,g) \left( \Vert v_i\Vert_p - \Vert v_i\Vert_{L^p(B_{2R})}\right)^2 - \frac{4a_n}{R^2}\Vert v_i\Vert_{L^2(B_{2R})}^2  +\mu_i\int_{B_R} f_i^2v_i^2 \vo_g\right)\\ 
\end{align*}

With $\Vert v_i\Vert_{L^s(U)}\to 0$ for $1\leq s\leq p$ on compact subsets $U\subset M$, $\Vert v_i\Vert_p\to 1$ and
\begin{align*} \int_{B_R} f_i^2v_i^2\vo_g & \leq \Vert f_i\Vert_n \Vert v_i\Vert_{L^p(B_R)}^2\leq \Vert v_i\Vert_{L^p(B_R)}^2 \to 0
\end{align*}
we obtain for all $R$ that 
\[ Q(M,g)\geq Q(M\setminus B_R,g).\]
That contradicts $\ov{Q}> Q$. Thus, $v\ne 0$.
\end{proof}

Now we can estimate

\[Q\leq \frac{\int_M vL_gv\vo_g}{\Vert v\Vert_p^2}\leq Q\frac{\int_M f^2v^2\vo_g}{\Vert v\Vert_p^2}\leq Q\frac{\Vert f\Vert_n^2 \Vert v\Vert_p^2}{\Vert v\Vert_p^2}\leq Q.\]

Hence, there is already equality. In particular, from the equality case in the used H\"older inequality we get $f^2=v^{p-2}$ and $1=\Vert f\Vert_n=\Vert v\Vert_p$. Smoothness of $v$ is obtained by standard local elliptic regularity theory. 
By the maximum principle one sees that $v$ is everywhere positive which concludes the proof of Theorem \ref{main}.\hfill $\Box$

Standard local elliptic regularity also gives that $v$ is locally in $C^{2,\alpha}$.

\begin{rem}[On the assumption on the scalar curvature]\hfill\\
In Theorem \ref{main} we assume that $\Vert (\s_g)_-\Vert_{L^{\frac{n}{2}}(g)}<\infty$. If the Yamabe invariant $Q(g)=-\infty$, this could never be true. But in general it can happen that even though $\Vert (\s_g)_-\Vert_{L^{\frac{n}{2}}(g)}=\infty$, $Q$ is finite and even positive. The easiest example is the standard hyperbolic space $\Hy^n$ which has constant negative scalar curvature, infinite volume but the Yamabe invariant of the standard sphere. From this we can even easily construct an example with finite volume: Firstly, we note that $\Vert (\s_g)_-\Vert_{L^{\frac{n}{2}}(g)}$ is scale invariant. Let us take a ball $B$ in the hyperbolic space with $\Vert (\s_g)_-\Vert_{L^{\frac{n}{2}}(g)}=1$ and then rescale it such that the rescaled ball $B_i$ has volume $i^{-2}$. If we consider the disjoint sum of the $B_i$, we obtain an example for a (disconnected) Riemannian manifold of finite volume and $\Vert (\s_g)_-\Vert_{L^{\frac{n}{2}}(g)}=\infty$.

We assume that $Q(g)=-\infty$ if and only if $\Vert (\s_g)_-\Vert_{L^{\frac{n}{2}}(\ov{g})}=\infty$ for all $\ov{g}\in [g]$. But unfortunately we still cannot prove this. Even if this is true, this alone does not help in our context since we need a complete metric of finite volume with $\Vert (\s_g)_-\Vert_{L^{\frac{n}{2}}(\ov{g})}< \infty$ which probably cannot be achieved in general.
\end{rem}

%
%

\section{On closed manifolds}

The method we used in Theorem \ref{main} for complete manifolds of finite volume allows to reprove the result on closed manifolds with positive Yamabe invariant.

\begin{thm}\label{main_closed}
Let $(M,g)$ be a closed $n$-dimensional Riemannian manifold with $0<Q<Q(S^n)$. Then, there is a smooth positive solution $v\in H_1^2$ of the Yamabe equation \eqref{EL-eq}.
\end{thm}

\begin{proof}
The proof in the closed case is essentially the same as the one presented in Section \ref{proof_of_main}. The only little difference occurs in the proof of Lemma \ref{lem2} where the cut-off function $\chi_R$ is introduced and $Q$ is estimated. We make the following change --  we take the smooth cut-off function $\eta_\epsilon$ introduced before in Lemma \ref{lem2}. Then with the same estimate as in Lemma \ref{lem2}, where $M\setminus B_\epsilon$ substitutes  $B_{2R}$ and $M\setminus B_{2\epsilon}$ replaces $B_R$, we obtain
\begin{align*} Q&\geq \lim_{i\to\infty} \left( Q(B_{2\epsilon},g)(\Vert v_i\Vert_p-\Vert v_i\Vert_{L^p(M\setminus B_\epsilon)})^2-\frac{4a_n}{\epsilon^2}\Vert v_i\Vert_{L^2(M\setminus B_\epsilon)}+\mu_i \int_{M\setminus B_{2\epsilon}} f_i^2v_i^2\vo_g \right)\\
&=Q(B_{2\epsilon},g)
\end{align*}
For $\epsilon$ small enough this gives a contradiction to $Q(M)<Q(S^n)$ due to Lemma \ref{local_comp}.
Thus, following the rest of the proof in Section \ref{proof_of_main} we obtain that $v$ is a smooth positive solution of $L_gv=Qv^{p-1}$ with $\Vert v\Vert_p=1$. Note that on closed manifolds the condition $\Vert (\s_g)_-\Vert_{L^\frac{n}{2}(g)}<\infty$ of Theorem \ref{main} is trivially fulfilled.
\end{proof}


\bibliographystyle{acm}
\bibliography{compl+fin+Yam}

\end{document}